\documentclass[12pt,a4paper,reqno]{amsart}

\usepackage[T1]{fontenc}
\usepackage{microtype, a4wide, textcomp,fbb}
\usepackage{
    amsmath,  amssymb,  amsthm,   amscd,
    gensymb,  graphicx, etoolbox, 
    booktabs, stackrel, mathtools    
}
\usepackage[usenames,dvipsnames]{xcolor}
\definecolor{darkblue}{rgb}{0,0,0.7}
\definecolor{darkred}{rgb}{0.7,0,0}
\usepackage[colorlinks=true, linkcolor=darkred, citecolor=darkblue, urlcolor=blue, pagebackref=true, breaklinks=true]{hyperref}
\usepackage[capitalise]{cleveref}
\usepackage{placeins}
\usepackage{relsize}
\usepackage{todonotes}
\usepackage{soul}
\usepackage{enumitem}%For adding label to enumeration

\usepackage{tikz}
\usetikzlibrary{arrows}
\usetikzlibrary{shapes}
\usepackage{float}
\usepackage{tabularx}
\usepackage{kantlipsum}
\usepackage{array}
\usepackage[margin=2.6cm]{geometry}
\newtheorem{proposition}{Proposition}[section]
\newtheorem{lemma}[proposition]{Lemma}
\newtheorem{theorem}[proposition]{Theorem}

\newtheorem{corollary}[proposition]{Corollary}

\theoremstyle{definition}
\newtheorem{remark}[proposition]{Remark}

\newtheorem{definition}[proposition]{Definition}

\newtheorem{innercustomthm}{Theorem}
\newenvironment{customthm}[1]
  {\renewcommand\theinnercustomthm{#1}\innercustomthm\itshape}
  {\endinnercustomthm}

\newtheorem{innercustomcor}{Corollary}
\newenvironment{customcor}[1]
  {\renewcommand\theinnercustomcor{#1}\innercustomcor\itshape}
  {\endinnercustomcor}

\tikzstyle{place}=[draw,circle,minimum size=1mm,inner sep=1pt,outer sep=-1.1pt,fill=black]

\usetikzlibrary{shapes.geometric}
\tikzstyle{places}=[draw,rectangle,minimum size=8pt,inner sep=0pt]
\tikzstyle{placesf}=[draw,rectangle,minimum size=5pt,inner sep=0pt]
\tikzstyle{placec}=[draw,circle,minimum size=8pt,inner sep=0pt]
\tikzstyle{placecf}=[draw,circle, minimum size=5pt,inner sep=0pt]

\def\K{\mathbb{K}}
\def\G{\mathcal{G}}

\def\p{\mathfrak{p}}
\def\q{\mathfrak{q}}
\def\D{\Delta}
\def\F{\mathcal{F}}
\def\N{\mathcal{N}}
\def\d{\mathrm{depth}}

\def\h{\mathrm{ht}}

\def\l{\langle}
\def\r{\rangle}
\def\x{\mathbf x}

\begin{document}

\title{Square-free powers of Cohen-Macaulay simplicial forests}

\author{Kanoy Kumar Das}
\address{Chennai Mathematical Institute, India}
\email{kanoydas0296@gmail.com; kanoydas@cmi.ac.in}

\author{Amit Roy}
\address{Chennai Mathematical Institute, India}
\email{amitiisermohali493@gmail.com}

\author{Kamalesh Saha}
\address{Department of Mathematics, SRM University-AP, Amaravati 522240, Andhra Pradesh, India}
\email{kamalesh.s@srmap.edu.in; kamalesh.saha44@gmail.com}


\vspace*{-0.4cm}

\begin{abstract}
     Let $I(\D)^{[k]}$ denote the $k^{\text{th}}$ square-free power of the facet ideal of a simplicial complex $\D$ in a polynomial ring $R$. Square-free powers are intimately related to the `Matching Theory' and `Ordinary Powers'. In this article, we show that if $\D$ is a Cohen-Macaulay simplicial forest, then $R/I(\D)^{[k]}$ is Cohen-Macaulay for all $k\ge 1$. This result is quite interesting since all ordinary powers of a graded radical ideal can never be Cohen-Macaulay unless it is a complete intersection. To prove the result, we introduce a new combinatorial notion called special leaf, and using this, we provide an explicit combinatorial formula of $\d(R/I(\D)^{[k]})$ for all $k\ge 1$, where $\D$ is a Cohen-Macaulay simplicial forest. As an application, we show that the normalized depth function of a Cohen-Macaulay simplicial forest is nonincreasing. 
\end{abstract}

\maketitle
\section{Introduction}

A classical way to treat the class of square-free monomial ideals is by identifying them as a Stanley-Reisner ideal of simplicial complexes. This approach has been instrumental in bridging several areas of mathematics and produced many fascinating results. It was Faridi \cite{Faridi2002} who first treated square-free monomial ideals as a facet ideal of simplicial complexes and introduced the notion of simplicial trees. In case of simple graphs, the class of trees is nothing but the one-dimensional simplicial trees. 

Over the years, many algebraic properties of edge ideals of trees have been extended to facet ideals of simplicial trees. For instance, Faridi \cite{Faridi2002} generalized the sliding depth property of edge ideals of trees proved by Simis-Vasconcelos-Villarreal \cite{SVV1994} to the facet ideals of simplicial trees. Moreover, these ideals have normal and Cohen-Macaulay Rees~ rings.

A central theme in combinatorial commutative algebra is understanding how properties of graded ideals extend to their powers. Many algebraic invariants exhibit predictable asymptotic behavior. A classical example is the work of Herzog, Hibi, and Zheng \cite{HHZ04}, who extended the linearity of minimal free resolutions of edge ideals of simple graphs to all of their ordinary powers. If we consider the Cohen-Macaulay property of a square-free monomial ideal, it is known that all of its ordinary powers are Cohen-Macaulay only when the ideal is a complete intersection ideal. This situation becomes much more intricate if we consider other kind of powers, such as symbolic powers \cite{RTY11}, or square-free powers \cite{MF2024}. This article primarily focuses on the square-free powers of facet ideals of simplicial forests.

Let $I\subseteq R=\K[x_1,\ldots , x_n]$ be a square-free monomial ideal, where $\K$ is a field.
The concept of square-free powers was introduced in \cite{BHZN} and later extended to all monomial ideals through matching powers \cite{ErFi1}. The motivation for studying the square-free powers comes from graph theory. The $k^{th}$ \textit{square-free power} of $I$, denoted by $I^{[k]}$, is the ideal of $R$ generated by all square-free monomials of the ideal $I^k$. Note that $I^{[k]}=\l 0\r$ for $k>>0$. Let $\mathcal{G}(I)$ denote the unique minimal monomial generating set of a monomial ideal $I$. Then one can observe that $\mathcal{G}(I^{[k]})$ is the set of square-free monomials contained in $\mathcal{G}(I^k)$. In other words, the minimal generating set $\mathcal{G}(I^{[k]})$ of the $k^{th}$ square-free power of $I$ has a one-to-one correspondence with the set of all possible $k$-matchings of the underlying simplicial complex. So far, numerous studies have been conducted in various directions on the square-free powers of edge ideals of graphs (see, for instance, \cite{CFL1,DRS12024,EHHS,EHHM12,ErHi1,FiHeHi,MF2024,SaS1}). However, to the best of our knowledge, the investigation of square-free powers of arbitrary square-free monomial ideals has been explored only in \cite{kamberi2024squarefreepowerssimplicialtrees}.

The primary goal of this article is to explore the Cohen-Macaulay property of square-free powers of facet ideals of simplicial forests. It is well-known that for any graded radical ideal $I$ in a polynomial ring $R$, $R/I^k$ is Cohen-Macaulay for all $k$ if and only if $I$ is a complete intersection (see, for instance, \cite{CowsikNori}). This naturally raises the question of whether the same holds true for square-free powers. Interestingly, this is not the case, as it has been shown in \cite{DRS12024} that all square-free powers of the edge ideals of Cohen-Macaulay forests are also Cohen-Macaulay. Meanwhile, in \cite{MF2024}, Ficarra and Moradi showed that among the classes of chordal, Cameron-Walker, and very well-covered graphs, the only graphs whose all square-free powers of edge ideals satisfy the Cohen-Macaulay property are Cohen-Macaulay forests and complete graphs. To the best of our knowledge, no such square-free monomial ideals, apart from the above-mentioned edge ideals of simple graphs, are known for which all square-free powers are Cohen-Macaulay. Identifying these classes of ideals remains a challenging problem. In this paper, we provide a broad class of such square-free monomial ideals, which are not necessarily equigenerated. In particular, as a main result of this paper, we show that all square-free powers of the facet ideal of a Cohen-Macaulay simplicial forest are Cohen-Macaulay. 

To achieve our main result, we explicitly compute the (Krull) dimension and depth of $R/I(\Delta)^{[k]}$ for a Cohen-Macaulay simplicial forest $\Delta$. In this context, we introduce the notion of special leaves in a simplicial complex and show their existence in Cohen-Macaulay simplicial forests. Here, we remark that the concept of special leaves might be useful in further studies of simplicial forests. The computation of $\d(R/I(\D)^{[k]})$ intricately uses the combinatorial characterization of Cohen-Macaulay simplicial forests and the presence of a special leaf, along with the standard homological techniques. The numerical formula of the depth of $R/I(\Delta)^{[k]}$ also helps us to analyze the behavior of the normalized depth function. The normalized depth function of square-free powers was introduced by Erey-Herzog-Hibi-Madani in \cite{EHHM12}, where they conjectured that this function is nonincreasing. Later, Fakhari disproved this conjecture in \cite{Fakhari_increasing}. However, the conjecture still remains open for square-free powers of edge ideals of graphs and holds affirmatively for several subclasses (see, for instance, \cite{CFL1, MF2024}). It follows from our numerical formula for the depth that the conjecture holds affirmatively in the case of Cohen-Macaulay simplicial forests.

The paper is structured as follows. In \Cref{Section 2}, we recall the definitions and basic properties of a Cohen-Macaulay simplicial forest. The primary objective of this section is to introduce the concept of a special leaf in a simplicial complex. Further, we show the existence of such a special leaf in a Cohen-Macaulay simplicial forest (see \Cref{special leaf existence}). In \Cref{Section 3}, we first prove two technical lemmas \Cref{colon lemma} and \Cref{contraction CM}. Then, we establish the formula for the Krull dimension of the corresponding quotient rings of the square-free powers of a Cohen-Macaulay simplicial forest (see \Cref{dimension lemma}). Finally, we derive the depth formula as follows.

\begin{customthm}{\ref{main theorem 1}}
    Let $\D$ be a Cohen-Macaulay simplicial forest on the vertex set $V(\D)$. Then 
    \[
    \d(R/I(\D)^{[k]})=|V(\D)|-\nu(\D)+k-1,
    \]
    where $\nu(\D)$ is the matching number of $\D$.
\end{customthm}
\noindent As an immediate application of the above theorem and the computation of dimension, we obtain the main result of the paper:

\begin{customcor}{\ref{main cor}}
        Let $\D$ be a Cohen-Macaulay simplicial forest. Then for all $k\ge 1$, $I(\D)^{[k]}$ is also Cohen-Macaulay.
\end{customcor}
\noindent Additionally, in \Cref{normalised depth}, we show that the normalized depth function of square-free powers of the facet ideal of a Cohen-Macaulay simplicial forest is nonincreasing.

\section{Cohen-Macaulay Simplicial Forests and Special Leaf}\label{Section 2}

In this section, we begin by reviewing the definition of simplicial forests and the combinatorial criteria for their Cohen-Macaulay property. The main purpose of this section is to introduce a new combinatorial concept in a simplicial complex called \textit{special leaf}. Also, we show the existence of such a special leaf in a Cohen-Macaulay simplicial forest, and this fact is one of the key ingredients in proving our main result in the subsequent section.

A \textit{simplicial complex} $\Delta$ over a set of vertices $V=\{x_1,\ldots ,x_n\}$ is a collection of subsets of $V$ such that if $F\in \Delta$ and $G\subseteq F$, then $ G\in \Delta$. An element of $\D$ is called a \textit{face} of $\D$, and the \textit{dimension} of a face $F$ of $\Delta$, denoted by $\mathrm{dim}(F)$, is defined as $|F|-1$, where $|F|$ is the number of vertices in $F$. The dimension of $\D$, denoted by $\dim(\D)$, is given by $\mathrm{max}\{\mathrm{dim}(F)\mid F\in \D\}$. %The faces of dimensions $0$ and $1$ are called \textit{vertices} and \textit{edges}, respectively.
The maximal faces of $\D$ under inclusion are called the \textit{facets} of $\D$. We denote the set of all facets of $\D$ by $\F(\D)$. If $\F(\D)=\{F_1,\ldots , F_m\}$, then we simply write $\D=\langle F_1,\ldots , F_m\rangle$. For each $F_i\in\F(\D)$, the set $\N_{\Delta}(F_i)=\{F_j\in\F(\Delta)\mid F_j\cap F_i\neq\emptyset\text{ and }j\neq i\}$ is called the set of {\it neighbors} of $F_i$. The set $\N_{\Delta}[F_i]=\N_{\Delta}(F_i)\cup\{F_i\}$ is called the {\it closed neighborhood} of $F_i$ in $\Delta$. In this article, by a \textit{subcomplex of $\D$}, we mean a simplicial complex $\D'$  where $\F(\D')\subseteq \F(\D)$ following \cite{Faridi2002}. Note that this differs from the standard notion of subcomplex commonly used in the literature.

\subsection{Simplicial forests and their facet ideals} In \cite{Faridi2002}, Faridi extended the notion of forests in graph theory to the realm of simplicial complexes and introduced the class of simplicial forests. A facet $F$ of a simplicial complex $\D$ is called a \textit{leaf} if either $F$ is the only facet of $\D$ or there exists a facet $G \in \D$ such that $G\neq F$ and $F \cap G \supseteq F \cap H$ for all facets $H \neq F$ of $\D$. Such a facet $G$ is then called a \textit{joint} of $F$ in $\D$. 
A vertex of a simplicial complex $\D$ is called a \textit{free vertex} if it belongs to exactly one facet of $\D$. Note that every leaf $F$ of $\D$ contains a \textit{free vertex}. A connected simplicial complex $\D$ is a \textit{simplicial tree} if every nonempty subcomplex of $\D$ has a leaf. A {\it simplicial forest} is a simplicial complex whose every connected component is a simplicial tree. It directly follows from the definition that if $\D$ is a simplicial tree, then any subcomplex of $\D$ is a simplicial forest.

Let $\D$ be a simplicial complex with the set of facets $\F(\D)$. A subset $M$ of $\F(\D)$ is called a {\it matching} of $\D$ if for any two distinct facets $F,F'\in M$, we have $F\cap F'=\emptyset$. The maximum cardinality of a set of matchings in $\D$ is called the {\it matching number} of $\D$ and is denoted by $\nu(\D)$. Let $V(\D)=\{x_1,\ldots,x_n\}$, and $R$ denote the polynomial ring $\K[x_1,\ldots , x_n]$, where $\K$ is a field. Sometimes $R$ is also denoted by $\mathbb K[V(\D)]$ in order to specify the relationship with the vertex set $V(\D)$. For a subset of variables $S\subseteq \{x_1,\ldots , x_n\}$, we will write $\mathbf{x}_S$ to denote the monomial $\prod_{x_i\in S}x_i$. Then the \textit{facet ideal} of $\D$, denoted by $I(\D)$, is the ideal generated by monomials corresponding to each facet of $\D$, i.e.,
\[
I(\D)=\langle \mathbf{x}_F \mid F \text{ is a facet of } \D \rangle.
\]

\subsection{Cohen-Macaulay simplicial forests} The Cohen-Macaulay property of a simplicial forest is combinatorially characterized in terms of the notion of grafting in \cite{Faridi2005}. In fact, the procedure called grafting of a simplicial complex was first introduced by Faridi \cite{Faridi2005}, which we describe below.

\begin{definition}\cite[Definition 7.1]{Faridi2005}\label{CM SF defn}
    A simplicial complex $\D$ is said to be a \textit{grafting} of a simplicial complex $\D'=\langle G_1,\ldots , G_s\rangle$ with the simplices $F_1,\ldots ,F_r$ (or we say that $\D$ is \textit{grafted}) if $\D=\langle F_1,\ldots ,F_r\rangle\cup \langle G_1,\ldots , G_s\rangle$ with the following properties:
    \begin{enumerate}
        \item $V(\D')\subseteq \cup_{i=1}^r F_i$,
        \item $F_1,\ldots ,F_r$ are all the leaves of $\D$,
        \item $\{ F_1,\ldots ,F_r\} \cap \{ G_1,\ldots , G_s\}=\emptyset$,
        \item For $i\neq j$, $F_i\cap F_j=\emptyset$,
        \item If $G_i$ is a joint of $\D$, then $\D\setminus \langle G_i \rangle$ is also grafted. Equivalently, for each $G_i\in \F(\Delta')$, $\D\setminus\l G_i\r$ is a grafted simplicial complex.
    \end{enumerate}
\end{definition}

One of the main results in this direction is the following combinatorial classification of the Cohen-Macaulay property of the facet ideal of a simplicial tree.

\begin{theorem}\label{CM forest theorem}\cite[Theorem 8.2, Corollary 7.8]{Faridi2005}
    Let $\D$ be a simplicial tree, and let $I(\D)$ denote its facet ideal in the polynomial ring $R$. Then $R/I(\D)$ is Cohen-Macaulay if and only if $\D$ is grafted.
\end{theorem}

Let $\Delta=\l F_1,\ldots,F_r\r\cup\l G_1,\ldots,G_s\r$ be a grafting of the simplicial complex $\Delta'=\l G_1,\ldots,G_s\r$. If for some $j\in[s]$ and $i\in[r]$, $G_j\subseteq F_i$, then we see that $\Delta$ can also be considered as a grafting of the simplicial complex $\Delta''=\l G_1,\ldots,\widehat{G_j},\ldots,G_s\r$. Thus, without loss of generality, we may assume that $F_1,\ldots,F_r,G_1,\ldots,G_s$ are all the facets of $\Delta$, i.e., $\F(\D)=\{F_1,\ldots,F_r,G_1,\ldots,G_s\}$. Such a presentation of $\Delta$ is said to be a \textit{minimal presentation}, and it is unique up to a permutation of the indices.

\begin{remark}\label{remark 2}
    From now onwards, when we say that $\Delta=\l F_1,\ldots,F_r\r\cup\l G_1,\ldots,G_s\r$ is a Cohen-Macaulay simplicial forest, we assume that $\Delta$ has the minimal presentation and $\Delta$ is a grafting of the simplicial forest $\Delta'=\l G_1,\ldots,G_s\r$ by the simplices $F_1,\ldots,F_r$. 
\end{remark}

The notion of a good leaf was introduced in \cite{CFS1}. However, the existence of such a leaf in an arbitrary simplicial forest was first proved in \cite{HHTX1}. 

\begin{definition}[Good leaf]
    A facet $F$ of a simplicial complex $\D$ is called a \textit{good leaf} if $F$ is a leaf of every subcomplex of $\D$ that contains the facet $F$. 
\end{definition}

The equivalent definition of a good leaf given below is well-known in the literature, but we include a short proof here for the sake of completeness.

\begin{proposition}\label{goof leaf prop}
    Let $\D$ be a simplicial complex and $F\in \D$ be a leaf with $\N_{\D}(F)=\{F_1,\ldots , F_m\}$. Then $F$ is a good leaf of $\D$ if and only if there is a permutation $\{j_1,\ldots,j_m\}$ of the set $\{1,\ldots,m\}$ such that $F\cap F_{j_1}\supseteq F\cap F_{j_2}\supseteq\dots\supseteq F\cap F_{j_m}$.
\end{proposition}
    
\begin{proof}
    The `only if' part of the proof follows from the definition of good leaf. Conversely, without loss of generality, assume that $F\cap F_{1}\supseteq F\cap F_{2}\supseteq\dots\supseteq F\cap F_{m}$. Let $\D'$ be a subcomplex of $\Delta$ containing $F$. If $F\cap G=\emptyset$ for each $F\neq G\in\F(\D')$, then $F$ is clearly a leaf of $\D'$. Otherwise, take $l=\mathrm{min}\{i\mid F_i\in \D'\}$. Then we have $F\cap F_l\supseteq F\cap G$ for all $G\in \F(\D')$. Therefore, $F$ is a leaf in $\D'$.
\end{proof}

 For a Cohen-Macaulay simplicial forest, it follows from \Cref{CM forest theorem} that each leaf is a good leaf.

\begin{lemma}\cite{Faridi2005}\label{leaf chain}
    Let $\Delta=\l F_1,\ldots,F_r\r\cup\l G_1,\ldots,G_s\r$ be a Cohen-Macaulay simplicial forest. Then for each $1\leq i\leq r$, $F_i$ is a good leaf of $\Delta$.
\end{lemma}

\begin{remark}\label{rmk: matching number}
For a Cohen-Macaulay simplicial forest $\Delta=\l F_1,\ldots,F_r\r\cup\l G_1,\ldots,G_s\r$, we have $\nu(\Delta)=r$. Indeed, since $F_i\cap F_j=\emptyset$ for each $i\neq j$, we have $\nu(\Delta)\ge r$. On the other hand, if $\{H_1,\ldots,H_m\}$ is a matching of $\Delta$ such that $\nu(\Delta)=m$, then by (1) in \Cref{CM SF defn}, for each $i\in[m]$, there exists some $j_i\in[r]$ such that $H_i\cap F_{j_i}\neq \emptyset$. Our aim is to show that $F_{j_i}\neq F_{j_l}$ for each $i\neq l$. On the contrary, suppose $F_{j_i}=F_{j_l}$. Since $H_i\cap H_l=\emptyset$, we have $H_i,H_l\in\{G_1,\ldots,G_s\}$. Therefore, by \Cref{goof leaf prop} and \Cref{leaf chain}, we have $F_{j_i}\cap H_i\subseteq F_{j_l}\cap H_i$ or $F_{j_l}\cap H_i\subseteq F_{j_i}\cap H_i$, a contradiction to the fact that $H_i\cap H_j=\emptyset$. Thus $m\le r$ and we have $\nu(\Delta)=r$.
\end{remark}

\subsection{Special leaf} While the presence of a good leaf in any simplicial forest is helpful in this context, calculating the depth of square-free powers of facet ideals of Cohen-Macaulay simplicial forests explicitly requires the introduction of a new concept called \textit{special leaf}.

\begin{definition}[Special leaf]
Let $\D$ be a simplicial complex with the set of facets $\F(\D)$. A leaf $F$ of $\D$ is said to be a \textit{special leaf} if the following holds:
    \[
    (H\cap H')\setminus F\neq \emptyset\text{ if and only if }H\cap H'\neq\emptyset,
    \]
    where $H,H'\in\F(\D)\setminus\{F\}$.
\end{definition}

\begin{figure}[h!]
        \centering
        \begin{tikzpicture}
        [scale=.65]
            \path
            (0, 0)   coordinate (x)
            (-0.5, 1) coordinate (y1)
            (0.5, 1) coordinate (y2)
            (1, 0.5)   coordinate (y3)
            (1, -0.5)   coordinate (y4)
            (-0.5, -1)   coordinate (y5)
            (0.5, -1)   coordinate (y6)
            %(-1,  -0.5) coordinate (y7)
            %(-1,  0.5)   coordinate (y8)
            %(-2,  0.5)   coordinate (z)
            (0, -1.5)   coordinate (D1);

            \fill [lightgray]         (x) -- (y1) -- (y2) -- (x) -- cycle;
            \fill [lightgray]           (x) -- (y3) -- (y4) -- (x) -- cycle;
            \fill [lightgray]           (x) -- (y5) -- (y6) -- (x) -- cycle;
           % \fill [lightgray]           (x) -- (y7) -- (y8) -- (x) -- cycle;

            \draw
            (x) -- (y1) -- (y2) -- (x) -- cycle
            (x) -- (y3) -- (y4) -- (x) -- cycle
            (x) -- (y5) -- (y6) -- (x) -- cycle;
            %(z) -- (y8)
            %(x) -- (y7) -- (y8) -- (x) -- cycle;

            \tikzstyle{vertex} = [draw,circle,fill=black,inner sep = 1.5pt]
            \node[vertex] [label = left: $x$] at (x) {};
            \node[vertex] [label = above: $y_1$] at (y1) {};
            \node[vertex] [label = above: $y_2$] at (y2) {};
            \node[vertex] [label = right: $y_3$] at (y3) {};
            \node[vertex] [label = right: $y_4$] at (y4) {};
            \node[vertex] [label = below: $y_5$] at (y5) {};
            \node[vertex] [label = below: $y_6$] at (y6) {};
            %\node[vertex] [label = left: $y_7$] at (y7) {};
            %\node[vertex] [label = left: $y_8$] at (y8) {};
            %\node[vertex] [label = left: $z$] at (z) {};
            \node[] [label = below: $\D_1$] at (D1) {};
        \end{tikzpicture}
        \hspace{6em}
        \begin{tikzpicture}
        [scale=.65]
            \path
            (0.5, 0)   coordinate (x1)
            (-0.5, 0.5)   coordinate (y3)
            (-1.2, 0.5)   coordinate (y4)
            (0.5, -1) coordinate (x2)
            ( -0.5, -0.5) coordinate (x3)
            (1.2, 0.75)   coordinate (y1)
            (1.2, -1.5)   coordinate (y2)
            (-1.2, -0.5)   coordinate (x4)
            (0, -2)   coordinate (D2);

            \fill [lightgray]         (x1) -- (x2) -- (x3) -- (x1) -- cycle;

            \draw
            (x1) -- (x2) -- (x3) -- (x1) -- cycle
            (x1) -- (y1)
            (x3) -- (y3)
            (x4) -- (y4)
            (x2) -- (y2)
            (x3) -- (x4);

            \tikzstyle{vertex} = [draw,circle,fill=black,inner sep = 1.5pt]
            \node[vertex] [label = right: $x_1$] at (x1) {};
            \node[vertex] [label = right: $x_2$] at (x2) {};
            \node[vertex] [label = below: $x_3$] at (x3) {};
            \node[vertex] [label = above: $y_3$] at (y3) {};
             \node[vertex] [label = above: $y_4$] at (y4) {};
            \node[vertex] [label = above: $y_1$] at (y1) {};
            \node[vertex] [label = below: $y_2$] at (y2) {};
            \node[vertex] [label = below: $x_4$] at (x4) {};
            \node[] [label = below: $\D_2$] at (D2) {};
        \end{tikzpicture}
        \caption{Simplicial trees and special leaf}
        \label{Figure 1}
    \end{figure}
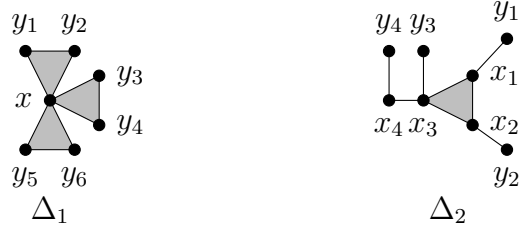
    Unlike a good leaf, a simplicial forest may not contain any special leaf. For instance, no leaf in the simplicial forest $\D_1$ in \Cref{Figure 1} is a special leaf. Indeed, if we set $F_1=\{x,y_1,y_2\}, F_2=\{x,y_3,y_4\}$ and $F_3=\{x,y_5,y_6\}$, then $(F_1\cap F_2)\setminus F_3=\emptyset$, but $F_1\cap F_2\neq\emptyset$ in $\D_1$. Similarly, $(F_1\cap F_3)\setminus F_2=\emptyset$, but $F_1\cap F_3\neq\emptyset$, and $(F_2\cap F_3)\setminus F_1=\emptyset$, but $F_2\cap F_3\neq\emptyset$ in $\D_1$. Now, consider the simplicial complex $\D_2$ in \Cref{Figure 1}, which is a Cohen-Macaulay simplicial tree. Note that the leaves $\{x_1,y_1\}, \{x_2,y_2\}$ and $\{x_4,y_4\}$ are special leaves. However, the leaf $\{x_3,y_3\}$ in $\Delta_2$ is not a special leaf since $\{x_3,x_4\}\cap \{x_1,x_2,x_3\}\neq \emptyset$ but $(\{x_3,x_4\}\cap \{x_1,x_2,x_3\})\setminus\{x_3,y_3\}=\emptyset$. Thus, not every leaf in a Cohen-Macaulay simplicial forest is a special leaf.
    \begin{remark}\label{CM special leaf remark}
        Let $\Delta=\l F_1,\ldots,F_r\r\cup\l G_1,\ldots,G_s\r$ be a Cohen-Macaulay simplicial forest with the set of leaves $\{F_1,\ldots,F_r\}$. Since $F_i\cap F_j=\emptyset$, it is easy to see that a leaf $F$ of $\Delta$ is a special leaf if and only if the following implication holds:
        \[
G_l\cap G_m\neq\emptyset \ \Rightarrow \ (G_l\cap G_m)\setminus F\neq \emptyset,
        \]
    where $l,m\in[s]$ and $l\neq m$.
    \end{remark}

We conclude this section by showing the existence of a special leaf in a Cohen-Macaulay simplicial forest.

\begin{lemma}\label{special leaf existence}
    Let $\Delta$ be a Cohen-Macaulay simplicial forest. Then, $\Delta$ contains at least one special leaf.
\end{lemma}

\begin{proof}
    Let $\Delta=\l F_1,\ldots,F_r\r\cup\l G_1,\ldots,G_s\r$ be a Cohen-Macaulay simplicial forest, and set $\D'=\l G_1,\ldots,G_s\r$. Take any good leaf of $\D'$, say $G_1\in \F(\D')$. Without loss of generality, let $\N_{\D'}[G_1]=\{G_1,\ldots , G_t\}$ for some $t\leq s$. Since $G_1$ is a good leaf, by \Cref{goof leaf prop}, we can assume, without loss of generality, that $G_1\cap G_2\supseteq G_1\cap G_3\supseteq \cdots \supseteq G_1\cap G_t$. Now, take any leaf $F$ of $\D$ such that $F\cap G_1\neq \emptyset$. We first observe that $\N_{\D}(F)\subseteq \N_{\D'}[G_1]$. Indeed, if $F\cap G_j\neq \emptyset$ for some $j\in [s]\setminus\{1\}$, then since $F$ is a good leaf of $\D$, we have either $F\cap G_1\supseteq F\cap G_j\neq \emptyset$, or $F\cap G_j\supseteq F\cap G_1\neq \emptyset$. Thus in any case, $G_1\cap G_j\neq \emptyset$, which in turn, shows that $G_j\in \N_{\D'}[G_1]$.
    
    First, we consider the case when $F\nsupseteq G_1\cap G_t$. We claim that, in this case, $F$ is a special leaf. Note that, by \Cref{CM special leaf remark}, it is enough to show that if $G_l\cap G_m\neq \emptyset$ for some distinct $l,m\in [s]$, then $(G_l\cap G_m)\setminus F\neq \emptyset$. 
    
    \noindent \textbf{Case-I:} Either $G_l\notin \N_{\D'}[G_1]$, or $G_m\notin \N_{\D'}[G_1]$. Without loss of generality, assume $G_l\notin \N_{\D'}[G_1]$. Then we have $G_l\cap F=\emptyset$ since $\N_{\D}(F)\subseteq \N_{\D'}[G_1]$. Hence, $(G_l\cap G_m)\cap F=\emptyset$, which in turn, implies that $(G_l\cap G_m)\setminus F\neq \emptyset$, where $i\in [s]$.
    
     \noindent \textbf{Case-II:} $G_l,G_m\in \N_{\D'}[G_1]$. Then, since $G_1$ is a good leaf in $\D'$, without loss of generality, we can assume that $G_1\cap G_l\supseteq G_1\cap G_m\supseteq G_1\cap G_t$. Since $F\nsupseteq G_1\cap G_t$, there is some $x\in G_1\cap G_t$ such that $x\notin F$. Observe that $x\in G_l\cap G_m$, and $x\notin F$. This shows that $x\in (G_l\cap G_m)\setminus F$, where $i\in [s]$. This completes the proof of the claim.

     \noindent
     Finally, if $F\supseteq G_1\cap G_t$, then take any other leaf $F'\neq F$ of $\D$ such that $F'\cap G_1\neq \emptyset$. Such a leaf $F'$ exists since $G_1$ is a facet of $\D$. In this case, proceeding as before, we have $\N_{\D}(F')\subseteq \N_{\D'}[G_1]$. Now since $F\cap F'=\emptyset$, we must have $F'\nsupseteq G_1\cap G_t$. Then, we can proceed in a similar way as in the previous case to show that $F'$ is a special leaf of $\D$. Thus, $\D$ contains at least one special leaf.
\end{proof}

\section{Square-free Powers of Cohen-Macaulay Simplicial Forests}\label{Section 3}

In this section, we explicitly derive the dimension and depth of square-free powers of the facet ideal of a Cohen-Macaulay simplicial forest. As a consequence, we get our main result, which states that all the square-free powers of a Cohen-Macaulay simplicial forest are also Cohen-Macaulay. Additionally, we prove that the normalized depth function of such an ideal is nonincreasing.

Before going to the main content of this section, we clarify some notations to avoid any confusion for the reader. For a square-free monomial ideal $I$, by $I^{[k]}$ we mean the $k^{\text{th}}$ square-free power of $I$, whereas $i\in[k]$ means the element $i$ is in the set $[k]:=\{1,2,\ldots,k\}$.

We begin with the following lemma, which shows that for the square-free powers of the facet ideal of a Cohen-Macaulay simplicial forest, taking colon by the monomial associated with a leaf behaves well. Here we assume that $I(\D)^{[k]}=R$ for each $k\le 0$.

\begin{lemma}\label{colon lemma}
    Let $\Delta$ be a Cohen-Macaulay simplicial forest, and let $F$ be a leaf of $\Delta$. Then for all $1\le k\le \nu(\Delta)$,
    \[
    (I(\Delta)^{[k]}:\x_F)=I(\Delta_1)^{[k-1]},
    \]
    where $\F(\Delta_1)=\F(\Delta)\setminus \N_{\Delta}[F]$. Moreover, $\Delta_1$ is also a Cohen-Macaulay simplicial forest.
\end{lemma}
\begin{proof}
    The $k=1$ case is easy to observe. Therefore, we may assume that $k\ge 2$. Let $\D=\l F_1,\ldots,F_r\r\cup\l G_1,\ldots,G_s\r$. Without loss of generality, assume that $F=F_1$, and $\N_{\Delta}(F_1)=\{G_1,\ldots,G_t\}$ for some $t\le s$. Take a minimal generator $\prod_{i=1}^k\x_{H_i}$ in $I(\Delta)^{[k]}$. Then $\{H_1,\ldots,H_k\}\subseteq \F(\Delta)$ forms a $k$-matching in $\Delta$. If $H_j=F_1$ for some $j\in[k]$, then $H_l\notin \N_{\Delta}(F_1)$ for each $l\in[k]\setminus\{ j\}$. Thus, in this case, we have
    \[
     \frac{\prod_{i=1}^k\x_{H_i}}{\gcd(\x_{F_1}, \prod_{i=1}^k\x_{H_i})}=\prod_{\underset{i\neq j}{i=1}}^k\x_{H_i}\in I(\Delta_1)^{[k-1]}.
    \]
    Thus, we may assume that $H_j\neq F_1$ for each $j\in[k]$. Now suppose $H_j\notin \N_{\Delta}(F_1)$ for each $j\in [k]$. Then, it is easy to see that
\[
    \frac{\prod_{i=1}^k\x_{H_i}}{\gcd(\x_{F_1}, \prod_{i=1}^k\x_{H_i})}=\prod_{i=1}^k\x_{H_i} \in I(\Delta_1)^{[k]}\subseteq I(\Delta_1)^{[k-1]}.
    \]
    Next, consider the case when $H_j\in \N_{\Delta}(F_1)$ for some $j\in[k]$. Observe that if $H_l\in \N_{\Delta}(F_1)$ for some $l\in[k]$ with $l\neq j$, then by \Cref{leaf chain}, we have either $H_l\cap F_1\subseteq H_j\cap F_1$ or $H_j\cap F_1\subseteq H_l\cap F_1$. In particular, $H_l\cap H_j\neq \emptyset$, a contradiction to the fact that $\{H_1,\ldots,H_k\}$ forms a $k$-matching in $\Delta$. Thus $H_l\notin \N_{\Delta}[F_1]$ for each $l\in[k]\setminus\{j\}$, and hence, $\prod_{\underset{i\neq j}{i=1}}^k\x_{H_i}\in I(\Delta_1)^{[k-1]}$. Consequently,
    \[
    \frac{\prod_{i=1}^k\x_{H_i}}{\gcd(\x_{F_1}, \prod_{i=1}^k\x_{H_i})}=(\x_{H_j\setminus F_1})\cdot \left(\prod_{\underset{i\neq j}{i=1}}^k\x_{H_i}\right)\in I(\Delta_1)^{[k-1]}.
    \]
Thus considering all the above cases, we have $(I(\Delta)^{[k]}:\x_{F_1})\subseteq I(\Delta_1)^{[k-1]}$. On the other hand, since $\F(\Delta_1)=\F(\Delta)\setminus \N_{\Delta}[F_1]$, we observe that if $H_1,\ldots,H_{k-1}\in\F(\D_1)$ forms a $(k-1)$-matching in $\D_1$, then $H_1,\ldots,H_{k-1},F_1$ forms a $k$-matching in $\D$. Hence, $I(\Delta_1)^{[k-1]}\subseteq (I(\Delta)^{[k]}:\x_{F_1})$. This completes the proof of the first part of the lemma.

It now remains to show that the subcomplex $\D_1=\l F_2,F_3,\ldots , F_r\r\cup \l G_{t+1}, G_{t+2},\ldots , G_s\r$ is also a Cohen-Macaulay simplicial forest. Since $\D_1$ is a subcomplex of $\D$, $\D_1$ is again a simplicial forest. Thus, it is enough to show that $\D_1$ is a grafted simplicial complex. By \Cref{leaf chain}, as $F_1$ is a good leaf, without loss of generality, we can assume that $F_1\cap G_1\supseteq F_1\cap G_2\supseteq \cdots \supseteq F_1\cap G_t$. Note that $G_1$ is a joint of $F_1$, and hence, $\D\setminus \l G_1\r$ is also grafted. Now, $G_2$ is a joint of $F$ in $\D\setminus \l G_1\r$ and then $\D\setminus \l G_1,G_2\r$ is grafted. Continuing in this way, we finally obtain that $\D\setminus \l G_1,G_2,\ldots , G_t\r$ is grafted. Observe that $\D\setminus \l G_1,G_2,\ldots , G_t\r=\l F_1\r\cup \D_1$, where $F_1\cap V(\D_1)=\emptyset$. Thus, $\D_1$ is also grafted, and this completes the proof.
\end{proof}

The notion of contraction of a simplicial complex is well-studied in the literature. We recall the definition and prove an auxiliary lemma below that will be essential in the computation of depth.

\begin{definition}[Contraction]
    Let $\Delta$ be a simplicial complex with $\F(\Delta)=\{F_1,\ldots,F_l\}$ and let $A\subseteq V(\Delta)$. The \textit{contraction of $\D$ on $A$} is a simplicial complex $\D_A$ on the vertex set $V(\D)\setminus A$ whose facets are given by
    \[
        \F(\D_A)=\{F_i\setminus A\mid F_i\in\F(\Delta)\text{ and } F_i\setminus A\nsupseteq F_j\setminus A\text{ for each }j\neq i \}.
    \]
\end{definition}

The following lemma is needed in the proof of \Cref{contraction CM}.

\begin{lemma}\label{colon cohen macaulay}
    Let $I$ be a monomial ideal in a polynomial ring $R$ and let $m\notin I$ be a monomial in $R$. If $R/I$ is Cohen-Macaulay, then so is $R/(I:m)$.
\end{lemma}
\begin{proof}
    By \cite[Corollary 1.3]{Rauf2010} we have $\d(R/I)\le \d(R/(I:m))$. On the other hand, $\d(R/(I:m))\le \dim(R/(I:m))\le \dim(R/I)=\d(R/I)$ since $R/I$ is Cohen-Macaulay. Thus $\d(R/(I:m))=\dim(R/(I:m))$ and hence $R/(I:m)$ is Cohen-Macaulay.
\end{proof}
Given a Cohen-Macaulay simplicial forest $\D$, one can derive that $\D_A$ is also a Cohen-Macaulay simplicial forest. However, in order to determine the grafting structure in $\D_A$, one needs to carefully use the grafting structure in $\D$. In the next lemma, we have explicitly determined the grafting structure of a special type of contraction in a Cohen-Macaulay simplicial forest that will be essential in the computation of depth in \Cref{main theorem 1}. 

\begin{lemma}\label{contraction CM}
    Let $\Delta=\l F_1,\ldots,F_r\r\cup\l G_1,\ldots,G_s\r$ be a Cohen-Macaulay simplicial forest. Take $A_{F_i}\subseteq \bigcap_{F_i\cap G_j\neq\emptyset}(F_i\cap G_j)$. Then $\Delta_{A_{F_i}}$ is also a Cohen-Macaulay simplicial forest. Moreover, the structure of $\D_{A_{F_i}}$ can be explicitly determined. 
\end{lemma}
\begin{proof}
Note that $A_{F_i}$ is not a facet of $\Delta$ and thus $\x_{A_{F_i}}\notin I(\Delta)$. Since $R/I(\Delta)$ is Cohen-Macaulay, by \Cref{colon cohen macaulay}, we have that $R/(I(\D):\x_{A_{F_i}})$ is also Cohen-Macaulay. Observe that $I(\Delta_{A_{F_i}})=(I(\D):\x_{A_{F_i}})$. Thus $R/I(\Delta_{A_{F_i}})$ is Cohen-Macaulay. We now proceed to explicitly determine the structure of $\D_{A_{F_i}}$ and show that $\D_{A_{F_i}}$ is a grafting of some simplicial forest.
     
     %The proof is by induction on $|V(\D')|$. If $|V(\D')|\le 3$, then it is easy to see that $\D_{A_{F_i}}$ is always a Cohen-Macaulay simplicial forest. Therefore, we may assume that $|V(\Delta')|\ge 4$. 
     
     Let $\Delta=\l F_1,\ldots,F_r\r\cup\l G_1,\ldots,G_s\r$ with $\Delta'=\l  G_1,\ldots,G_s\r$ and without loss of generality, we take $F_i=F_1$. If $A_{F_1}=\emptyset$, then $\Delta_{A_{F_1}}=\Delta$ and thus $\D_{A_{F_1}}$ is a grafting of $\D'$. Therefore, we may assume that $A_{F_1}\neq \emptyset$. For simplicity of notation, let us write $A=A_{F_1}$. We now describe a procedure to find a minimal presentation of $\Delta_{A}$ as follows.\par

     \begin{itemize}
         \item[$\bullet$] Without loss of generality, let us assume that $\N_{\D}(F_1)=\{G_1,\ldots,G_n\}$ for some $n\leq s$. Observe that $G_l\setminus A=G_l$ for each $l>n$, and $F_j\setminus A=F_j$ for each $j\neq 1$. Moreover, for each distinct $i,j\in [n]$, $G_i\setminus A\nsupseteq G_j\setminus A$. Furthermore, for any $i\in [n]$, $F_j\setminus A\nsubseteq G_i\setminus A$ and $G_l\setminus A\nsubseteq G_i\setminus A$, where $j\in[r]$ and $l>n$. Thus, $G_i\setminus A\in\F(\Delta_{A})$ for each $i\in[n]$. 
         
         \item[$\bullet$] Without loss of generality, let for each $n+1\le l\le m$, $G_l=G_l\setminus A\supseteq G_i\setminus A$ for some $i\in [n]$; and for each $l>m$ and $i\in [n]$, $G_l=G_l\setminus A\nsupseteq G_i\setminus A$, where $m>n$ is an integer. In this case, we have $G_l=G_l\setminus A\notin\F(\Delta_{A})$ for each $n+1\le l\le m$.

         \item[$\bullet$] Since $G_l\setminus A\nsupseteq F_i\setminus A$ for each $i\in[r]$, we see that $G_l\setminus A\in\F(\Delta_{A})$ for each $l>m$.

         \item[$\bullet$] Observe that $F_1\setminus A\nsupseteq G_i\setminus A$ for every $i\in[s]$, and hence, $F_1\setminus A\in \F(\D_A)$.

         \item[$\bullet$] Without loss of generality, let $2\le q\le r$ be an integer such that for each $2\le t\le q$, $F_t\setminus A\supseteq G_i\setminus A$ for some $i\in[n]$; and for each $t>q$ and $i\in [n]$, $F_t\setminus A\nsupseteq G_i\setminus A$. In this case, $F_t\setminus A\notin \F(\Delta_{A})$ for each $2\leq t\le q$.

         \item[$\bullet$] Moreover, since for each $j\in [r]$, $F_j\setminus A\nsupseteq G_l\setminus A=G_l$ for each $l>m$, we have that $F_t=F_t\setminus A\in\F(\Delta_{A})$ for each $t>q$.
     \end{itemize}
     
           Next, we proceed to show that for each $2\le t\le q$, there exists a unique $i_t\in [n]$ such that $F_t\setminus A\supseteq G_{i_t}\setminus A$. 
           
           \begin{itemize}
               \item[$\bullet$] Observe that $F_t\setminus A=F_t$ and if $F_t\supseteq G_{i_t}\setminus A$ and $F_t\supseteq G_{j_t}\setminus A$ for two distinct $i_t,j_t\in[n]$, then $F_t\cap G_{i_t}=G_{i_t}\setminus A$ and $F_t\cap G_{j_t}=G_{j_t}\setminus A$. Hence, by \Cref{leaf chain}, either $G_{i_t}\setminus A\subseteq G_{j_t}\setminus A$ or $G_{i_t}\setminus A\supseteq G_{j_t}\setminus A$. Thus, either $G_{i_t}\subseteq G_{j_t}$ or $G_{i_t}\supseteq G_{j_t}$, a contradiction to the fact that $G_i\in\F(\Delta)$ for each $i\in[n]$. Also, since $F_i\cap F_j=\emptyset$ for each $i\neq j$, we see that if for some $p\in [n]$, $G_p\setminus A\subseteq F_i$ and $G_p\setminus A\subseteq F_j$, then we must have $i=j$.
           \end{itemize}
            Without loss of generality, we may assume that $F_t\setminus A\supseteq G_{t-1}\setminus A$ for each $2\le t\le q$. Then there exists non-negative integers $q,n$ and $m$ such that $ q\le \min\{r-1,n\}$, $ n\le s$, $n< m\le s$, and the minimal presentation of $\Delta_{A}$ is 
    \begin{align}\label{equation 52}
    \Delta_{A}=\l F_1\setminus A,G_1\setminus A,\ldots,G_{q-1}\setminus A,F_{q+1},\ldots,F_r,G_{q}\setminus A,\ldots,G_n\setminus A,G_{m+1},\ldots,G_s\r.
    \end{align}
    Here we remark that if $A\subsetneq \bigcap_{F_1\cap G_j\neq\emptyset}(F_1\cap G_j)$, then no such $q$ exists. For simplicity of notation, let us write $F'_i=F_i\setminus A$ and $G'_i=G_i\setminus A$ for all $i$. We now proceed to show that $\Delta_A$ is a grafted simplicial forest.
    
    \noindent
\textbf{Claim 1}: The simplicial complex $\Delta'_{A}=\l G'_{q},\ldots,G'_n,G'_{m+1},\ldots,G'_s\r$ is a simplicial forest.

\noindent
\textit{Proof of Claim 1}: Let $\{G'_{j_1},\ldots,G'_{j_t}\}\subseteq \{G'_{q},\ldots,G'_n,G'_{m+1},\ldots,G'_{s}\}$ be an arbitrary collection of facets of $\D'_A$. We have to show the subcomplex $\l G'_{j_1},\ldots,G'_{j_t}\r$ of $\D'_A$ has a leaf. Consider the simplicial complex $\l G_{j_1},\ldots,G_{j_t}\r$, which is a subcomplex of $\D'$. Since $\D'$ is a simplicial forest, the subcomplex $\l G_{j_1},\ldots,G_{j_t}\r$ has a leaf, say $G_{j_1}$. Then there exists $i\in\{2\ldots, t\}$ such that $G_{j_1}\cap G_{j_{i}}\supseteq G_{j_1}\cap G_{j_k}$ for all $k\in\{2,\ldots,t\}\setminus\{i\}$. This implies $G'_{j_1}\cap G'_{j_{i}}\supseteq G'_{j_1}\cap G'_{j_k}$ for all $k\in\{2,\ldots,t\}\setminus\{i\}$. Hence, $G'_{j_1}$ is a leaf of the subcomplex $\l G'_{j_1},\ldots,G'_{j_t}\r$ of $\D'_A$, and consequently, $\D'_A$ is a simplicial forest.

\noindent
\textbf{Claim 2}: $\D_A$ is a grafting of $\D'_A$ with the simplices $F'_1,G'_1,\ldots,G'_{q-1},F'_{q+1},\ldots,F'_r$.

\noindent
\textit{Proof of Claim 2}: In order to prove Claim 2, we verify conditions (1)-(5) in \Cref{CM SF defn}. Since for any two leaves $F_i$ and $F_j$ of $\Delta$, $F_i\cap F_j=\emptyset$, we observe that for each $a\in\{1,q+1,\ldots,r\}$ and $b\in\{1,\ldots,q-1\}$, $F_a'\cap G_b'=\emptyset$. Moreover, it is easy to see that for each $i,j\in\{1,q+1,\ldots,r\}$, $F_i'\cap F_j'=\emptyset$. Furthermore, for each $a,b\in [q-1]$, $G_a'\cap G_b'=\emptyset$ since $G_a'\subseteq F_{a+1}$ and $G_b'\subseteq F_{b+1}$. Thus, the condition (4) in \Cref{CM SF defn} is verified. Next, for each $i\in\{1,q+1,\ldots, r\}$, using \Cref{leaf chain}, we obtain a chain of subsets
\[
F_i\cap G_{i_1}\supseteq F_i\cap G_{i_2}\supseteq\cdots\supseteq F_i\cap G_{i_l},
\]
where $\N_{\D}(F_i)=\{G_{i_1},\ldots, G_{i_l}\}$. Consequently, there exists a subset $\{j_1,\ldots,j_p\}\subseteq \{i_1,\ldots,i_l\}$ such that 
\begin{align}\label{equation 1}
F_i'\cap G_{j_1'}\supseteq F_i'\cap G_{j_2'}\supseteq\cdots\supseteq F_i'\cap G_{j_p}',
\end{align}
where $\N_{\Delta_A}(F_i')=\{G_{j_1}',\ldots,G_{j_p}'\}$. Thus for each $i\in\{1,q+1,\ldots,r\}$, $F_i'$ is a leaf of $\Delta_A$. Now fix some $i\in[q-1]$ so that $F_{i+1}\supseteq G_i'$. Observe that if $G_j'\in \N_{\Delta_A}(G_i')$, then $j\in\{q,\ldots,n,m+1,\ldots,s\}$ and $G_j\in \N_{\D}(F_{i+1})$. By \Cref{leaf chain}, we again have a chain of subsets
\[
F_{i+1}\cap G_{l_1}\supseteq F_{i+1}\cap G_{l_2}\supseteq\dots\supseteq F_{i+1}\cap G_{l_k},
\]
where $\N_{\D}(F_{i+1})=\{G_{l_1},\ldots,G_{l_k}\}$, and $G_{l_a}=G_i$ for some $a\in[k]$. Consequently, we have the following chain of subsets
\begin{align}\label{equation 41}
F_{i+1}'\cap G_{l_1}'\supseteq F_{i+1}'\cap G_{l_2}'\supseteq\dots\supseteq F_{i+1}'\cap G_{l_k}'.
\end{align}
Note that for each $1\le b\le a$, $F_{i+1}\cap G_{l_b}\supseteq F_{i+1}\cap G_{l_a}=G_{l_a}'$, and thus $G_{l_b}'\supseteq G_{l_a}'$. Consequently, $G_{l_b}'\notin\F(\Delta_A)$ if $b<a$. In this case, intersecting $G_{l_a}'$ with each member of the chain in \Cref{equation 41} we obtain
\[
G_{l_a}'\cap G_{l_a+1}'\supseteq \dots\supseteq G_{l_a}'\cap G_{l_k}',
\]
where $\N_{\D_A}(G_{l_a}')\subseteq \{G_{l_a+1}',\ldots,G_{l_k}'\}$. Thus we see that $G_i'$ is a leaf of $\D_{A}$ for each $i\in[q-1]$. Therefore, the condition (2) in \Cref{CM SF defn} is verified. 

To verify the condition (1) in \Cref{CM SF defn}, let us consider a vertex $x\in G_j'$, where $j\in\{q,\ldots,n,m+1,\ldots,s\}$. Then $x\in G_j$. If $x\in F_i$ for some $i\in\{1,q+1,\ldots,r\}$, then it is easy to see that $x\in F_i'$. Now suppose $x\in F_i$ for some $i\in\{2,\ldots,q\}$. Note that $F_i\cap G_{i-1}\neq\emptyset$, by construction. Thus, in this case, by \Cref{leaf chain}, either $F_i\cap G_j\supseteq F_i\cap G_{i-1}$ or $F_i\cap G_j\subseteq F_i\cap G_{i-1}$. Observe that if $F_i\cap (G_j\setminus A)=F_i\cap G_j\supseteq F_i\cap G_{i-1}=G_{i-1}'$, then $G_j'\supseteq G_{i-1}'$, a contradiction to the fact that $G_j'\in\F(\D_A)$. Thus we must have $F_i\cap G_j\subseteq F_i\cap G_{i-1}=G_{i-1}'$ and consequently, $x\in G_{i-1}'$. Thus, condition (2) in \Cref{CM SF defn} is verified. Also, by the minimal representation of $\D_A$ in \Cref{equation 52}, we see that condition (3) in \Cref{CM SF defn} is also verified. 

To prove the condition (5) in \Cref{CM SF defn}, first note that if $\D'$ contains exactly one facet $G_1$, then it follows from the definition of grafting that $\D_A\setminus \l G_1'\r=(\D\setminus \l G_1\r)_A$ is a grafted simplicial complex.
Now suppose $\D'$ contains more than one facet with $\F(\D')=\{G_1,\ldots,G_s\}$, as before. Further, suppose for $i\in\{q,\ldots,n\}$, we have $\{G_{i_1},\ldots,G_{i_p}\}\subseteq \F(\D')\setminus\{G_i\}$ such that $G_i\setminus A\subseteq G_{i_j}\setminus A$, for each $j\in[p]$. Then for each $i\in\{q,\ldots,n,m+1,\ldots,s\}$, we have
\begin{align*}
    \D_A\setminus \l G_i'\r=\begin{cases}
        (\D\setminus \l G_i\r)_A&\text{ if }i\in\{m+1,\ldots,s\},\\
        (\D\setminus\l G_i,G_{i_1},\ldots,G_{i_p}\r)_A&\text{ if }i\in\{q,\ldots,n\}.
    \end{cases}
\end{align*}
Here one can see that $A\subseteq \bigcap_{\underset{j\neq i}{F_1\cap G_j\neq\emptyset}}(F_1\cap G_j)$. Thus, using the fact that $\D\setminus \l G_i\r$ is grafted for each $i\in[s]$, and by the induction on a number of facets of $\D'$, we have that $\D_A\setminus\l G_i'\r$ is again a grafted simplicial forest. This completes the proof of Claim 2 and, subsequently, the proof of this lemma.
\end{proof}

We now proceed to determine the Krull dimension of $R/I(\Delta)^{[k]}$, where $\D$ is a Cohen-Macaulay simplicial forest.

\begin{theorem}\label{dimension lemma}
    Let $\Delta$ be a Cohen-Macaulay simplicial forest. Then for each $1\le k\le \nu(\Delta)$,
    \[
    \dim(R/I(\Delta)^{[k]})=|V(\Delta)|-\nu(\D)+k-1.
    \]
\end{theorem}
\begin{proof} Let $\D=\l F_1,\ldots,F_r\r\cup\l G_1,\ldots,G_s\r$ be a Cohen-Macaulay simplicial forest. Then by \Cref{rmk: matching number}, $\nu(\D)=r$. To prove the formula, it is enough to show that $\h(I(\Delta)^{[k]})=r-k+1$. Let $\p$ be a minimal prime ideal containing $I(\Delta)^{[k]}$. Recall that, $V(\D)\subseteq\cup_{i=1}^r F_i$. If $\h(\p)\le r-k$, then there exist leaves $F_{i_1},\ldots,F_{i_k}$ of $\Delta$ such that $\p\cap F_{i_j}=\emptyset$ for each $1\le j\le k$. In this case $\prod_{j=1}^k\x_{F_{i_j}}\in I(\Delta)^{[k]}$ but $\prod_{j=1}^k\x_{F_{i_j}}\notin\p$, a contradiction. Hence $\h(\p)\ge r-k+1$ for any minimal prime ideal containing $I(\Delta)^{[k]}$. Now we proceed to construct a minimal prime ideal $\q$ containing $I(\Delta)^{[k]}$ such that $\h(\q)=r-k+1$. For each $i\in[r]$, let
    \begin{align*}
        \mathcal{A}_i=\begin{cases}
            F_i&\text{ if }F_i\cap G_j=\emptyset\text{ for all }j,\\
            \cap_{F_i\cap G_j\neq\emptyset}(F_i\cap G_j)&\text{ if }F_i\cap G_j\neq\emptyset\text{ for some }j.
        \end{cases}
    \end{align*}
    By \Cref{leaf chain}, we have that $\mathcal A_i\neq\emptyset$ for each $i\in [r]$. Choose some $x_i\in\mathcal A_i$. Using the above description of the sets $\mathcal{A}_i$, we construct the prime ideal 
    \[
    \q=\l x_i\in \mathcal A_i\mid i\in [r-k+1]\r.
    \]
    Then $\h(\q)=r-k+1$ and thus it is enough to show that $I(\Delta)^{[k]}\subseteq \q$. Let $\prod_{i=1}^k\x_{H_i}\in\G(I(\Delta)^{[k]})$. Then for each $i\in[k]$ there exists a leaf $F_{t_i}$ of $\Delta$ such that $H_i\cap F_{t_i}\neq \emptyset$. We proceed to show that $F_{t_i}\neq F_{t_j}$ for each $i\neq j$. Observe that if $H_i$ is a leaf of $\Delta$ for some $i$, then there is nothing to show since $H_i,H_j$ forms a matching in $\Delta$. Now suppose $H_i=G_{t_i}$ and $H_j=G_{t_j}$ so that $G_{t_i}\cap F_{t_i}\neq\emptyset$ and $G_{t_j}\cap F_{t_j}\neq\emptyset$. Moreover, we assume $F_{t_i}=F_{t_j}$. In this case by \Cref{leaf chain}, $G_{t_i}\cap G_{t_j}\neq\emptyset$, a contradiction. Thus $F_{t_i}\neq F_{t_j}$ for each $i\neq j$. It is easy to see that $\mathcal A_{t_i}\subseteq H_i\cap F_{t_i}$, and in particular, $x_i\in H_i\cap F_{t_i}$ for each $i\in [k]$. Thus $\prod_{i=1}^k\x_{H_i}\in \q$ and this completes the proof.
\end{proof}

We are now ready to prove one of the main results of this article, where we establish the combinatorial formula for the depth of the $k^{\text{th}}$ square-free powers of the facet ideal of a Cohen-Macaulay simplicial forest.

\begin{theorem}\label{main theorem 1}
    Let $\Delta$ be a Cohen-Macaulay simplicial forest with the vertex set $V(\Delta)$. Let $R=\K[x_i\mid x_i\in V(\Delta)]$ be the polynomial ring over the field $\K$. Then, for each $1\le k\le\nu(\Delta)$, 
    \begin{enumerate}
        \item[(i)] $\d(R/I(\Delta)^{[k]})=|V(\Delta)|-\nu(\D)+k-1,$
        \item[(ii)] $\d(R/(I(\Delta)^{[k]}+\l\x_{F}\r))= |V(\Delta)|-\nu(\D)+k-1$, where $F$ is a special leaf of $\Delta$.
    \end{enumerate}

\end{theorem}

\begin{proof}
    We prove (i) and (ii) simultaneously by induction on $|V(\Delta)|$. If $|V(\Delta)|\le 2$, then it is easy to verify (i) and (ii). Indeed, if $|V(\D)|=1$, then $k=1$ and $I(\D)$ is generated by the variable in $R$, and the above inequalities hold. If $|V(\D)|=2$, then the only possibilities are $I(\D)=\l x_1,x_2\r$ or $I(\D)=\l x_1x_2\r$. When $I(\D)=\l x_1,x_2\r$, the possibilities of $k$ are $1$ and $2$, and in case $I(\D)=\l x_1x_2\r$ the only possibility of $k$ is $1$. Thus, one can observe that in all the cases the above equalities hold. Therefore, we may assume that $|V(\Delta)|\ge 3$.
    In this case, if $k=1$, then $I(\Delta)^{[1]}+\l\x_F\r=I(\Delta)$ for any special leaf $F$ of $\Delta$, and thus using \Cref{dimension lemma}, we have 
    \begin{align*}
    \d(R/I(\Delta)^{[1]})&=\d(R/(I(\Delta)^{[1]}+\l\x_F\r))\\
    &=\dim(R/I(\Delta))\\
    &=|V(\Delta)|-\nu(\D).    
    \end{align*}
     In view of this, from now onward, we assume that $k\ge 2$. Let $\Delta=\l F_1,\ldots,F_r\r\cup\l G_1,\ldots,G_s\r$ be a Cohen-Macaulay simplicial forest with $\Delta'=\l  G_1,\ldots,G_s\r$ and $\Delta$ is a grafting of $\Delta'$. Then by \Cref{rmk: matching number}, $\nu(\D)=r$. Now, by \Cref{special leaf existence}, $\Delta$ contains a special leaf. Without loss of generality, let $F_1$ be a special leaf of $\Delta$.

     \noindent
\textbf{Claim }: $\d(R/(I(\Delta)^{[k]}+\l\x_{F_1}\r))\ge |V(\Delta)|-r+k-1$.

\noindent
\textit{Proof of Claim}: 
    First, we consider the case when $F_1\cap G_i=\emptyset$ for each $i\in [s]$. In this case, $I(\Delta)^{[k]}+\l \x_{F_1}\r=I(\Delta_1)^{[k]}+\l \x_{F_1}\r$, where $\F(\Delta_1)=\F(\Delta)\setminus\{F_1\}$. It is easy to see that $\Delta_1$ is a Cohen-Macaulay simplicial forest with $\vert V(\D_1)\vert=\vert V(\D)\vert - \vert F_1\vert<\vert V(\D)\vert$. Therefore, by the additivity of depth \cite[Theorem 3.1.34]{RHV} and the induction hypothesis, we have
    \begin{align*}
        \d(R/(I(\Delta)^{[k]}+\l \x_{F_1}\r))&= \d(R/(I(\Delta_1)^{[k]}+\l \x_{F_1}\r))\\&=\d(\K[V(\D_1)]/I(\D_1)^{[k]})+\d(\K[F_1]/\l \x_{F_1}\r)\\ &=|V(\Delta_1)|-(r-1)+k-1+|F_1|-1\\
        &=|V(\Delta)|-r+k-1.
    \end{align*}
     Now suppose $F_1\cap G_i\neq\emptyset$ for some $i\in[s]$. Without loss of generality, let $\N_{\Delta}(F_1)=\{G_1,\ldots,G_n\}$ and let $A_{F_1}=\cap_{F_1\cap G_j\neq\emptyset}(F_1\cap G_j)$. By \Cref{leaf chain}, $A_{F_1}\neq \emptyset$ and by \Cref{contraction CM}, $\Delta_{A_{F_1}}$ is a Cohen-Macaulay simplicial forest. For the rest of the proof, we follow the notations introduced as in \Cref{contraction CM}. Thus there exists non-negative integers $q,n$ and $m$ such that $ q\le \min\{r-1,n\}$, $ n\le s$, $n< m\le s$, and %the minimal presentation of $\Delta_{A_{F_1}}$ is 
    \begin{equation}\label{Equation presentation}
    \begin{split}
    \Delta_{A_{F_1}}=\l F_1\setminus A_{F_1},&G_1\setminus A_{F_1},\ldots,G_{q-1}\setminus A_{F_1},F_{q+1},\ldots,F_r\r\\
    &\l G_{q}\setminus A_{F_1},\ldots,G_n\setminus A_{F_1},G_{m+1},\ldots,G_s\r,
   \end{split}
    \end{equation}
where $\{F_1\setminus A_{F_1},G_1\setminus A_{F_1},\ldots,G_{q-1}\setminus A_{F_1},F_{q+1},\ldots,F_r\}$ are all the leaves of $\D_{A_{F_1}}$.

    \noindent
\textbf{Subclaim 1}: $(I(\Delta)^{[k]}+\l \x_{F_1}\r):\x_{A_{F_1}}=I(\Delta_{A_{F_1}})^{[k]}+\l\x_{F_1\setminus A_{F_1}}\r$. 

\noindent
\textit{Proof of Subclaim 1}: Let $\prod_{i=1}^k\x_{H_i}\in \G(I(\Delta)^{[k]})$ such that $\x_{F_1}\nmid \prod_{i=1}^k\x_{H_i}$. Then $\{H_1,\ldots,H_k\}$ is a $k$-matching of $\Delta$ and $F_1\neq H_i$ for each $i\in[k]$. Note that, if $H_i\cap A_{F_1}\neq\emptyset$, then $H_i\in \N_{\Delta}(F_1)$. Thus $|\{i\mid H_i\cap A_{F_1}\neq\emptyset\}|\le 1$. Hence
\begin{align*}
     \frac{\prod_{i=1}^k\x_{H_i}}{\gcd(\x_{A_{F_1}}, \prod_{i=1}^k\x_{H_i})}=\begin{cases}
         \x_{H_j\setminus A_{F_1}}\cdot \prod_{\underset{i\neq j}{i=1}}^k\x_{H_i}&\text{ if }H_j\cap {A_{F_1}}\neq\emptyset\text{ for some }j\in[k],\\
         \prod_{i=1}^k\x_{H_i}&\text{ if }H_i\cap A_{F_1}=\emptyset\text{ for all }i\in[k].
     \end{cases}  
    \end{align*}
  One can observe from \Cref{Equation presentation} that if $H_i\cap A_{F_1}=\emptyset$, then either $H_i\in \F(\Delta_{A_{F_1}})$ or there exists some $H\in \F(\Delta_{A_{F_1}})$ such that $H\subseteq H_i$. Moreover, if $H_i\cap A_{F_1}\neq\emptyset$, then $H_i\setminus A_{F_1}\in\F(\Delta_{A_{F_1}})$. Thus using the fact that $\{H_1,\ldots,H_k\}$ forms a $k$-matching in $\Delta$, we obtain $(I(\Delta)^{[k]}+\l \x_{F_1}\r):\x_{A_{F_1}}\subseteq I(\Delta_{A_{F_1}})^{[k]}+\l\x_{F_1\setminus A}\r$. 

  Let $\prod_{i=1}^k\x_{H_i\setminus A_{F_1}}\in \G(I(\Delta_{A_{F_1}})^{[k]})$ such that $\x_{F_1\setminus A_{F_1}}\nmid \prod_{i=1}^k\x_{H_i\setminus A}$. In that case $H_i\in\F(\Delta)$ and $H_i\neq F_1$ for each $i\in [k]$. Also, since $F_1$ is a special leaf of $\Delta$, we have $H_i\cap H_j=\emptyset$ for each distinct $i,j\in[k]$. Thus $\{H_1,\ldots,H_k\}$ forms a $k$-matching in $\Delta$. Since $|\{i\mid H_i\cap A_{F_1}\}|\le 1$ we see that 
  \begin{align*}
      \x_{A_{F_1}}\cdot\prod_{i=1}^k\x_{H_i\setminus A_{F_1}}=\begin{cases}
          \prod_{i=1}^k\x_{H_i}&\text{ if }H_j\cap {A_{F_1}}\neq\emptyset\text{ for some }j\in[k],\\
          \x_{A_{F_1}}\cdot \prod_{i=1}^k\x_{H_i}&\text{ if }H_i\cap A_{F_1}=\emptyset\text{ for all }i\in[k].
      \end{cases}
  \end{align*}
  In particular, $\x_{A_{F_1}}\cdot\prod_{i=1}^k\x_{H_i\setminus A_{F_1}}\in I(\Delta)^{[k]}$. This completes the proof of Subclaim 1.

  \noindent
\textbf{Subclaim 2}: $F_1\setminus A_{F_1}$ is a special leaf of the Cohen-Macaulay simplicial forest $\Delta_{A_{F_1}}$.

\noindent
\textit{Proof of Subclaim 2}: By \Cref{contraction CM}, $\D_{A_{F_1}}$ is a Cohen-Macaulay simplicial forest, and $\{F_1\setminus A_{F_1},G_1\setminus A_{F_1},\ldots,G_{q-1}\setminus A_{F_1},F_{q+1},\ldots,F_r\}$ is the set of all the leaves of $\Delta_{A_{F_1}}$. By \Cref{CM special leaf remark}, it is enough to show that $(G_{t_1}\setminus A_{F_1})\cap(G_{t_2}\setminus A_{F_1})\neq\emptyset$ implies $((G_{t_1}\setminus A_{F_1})\cap(G_{t_2}\setminus A_{F_1}))\setminus (F_1\setminus A_{F_1})\neq\emptyset$, where $t_1,t_2\in \{q,\ldots,n,m+1,\ldots,s\}$ and $t_1\neq t_2$. Now suppose $(G_{t_1}\setminus A_{F_1})\cap(G_{t_2}\setminus A_{F_1})\neq\emptyset$. Hence $G_{t_1}\cap G_{t_2}\neq \emptyset$. Since $F_1$ is a special leaf we have that $(G_{t_1}\cap G_{t_2})\setminus F_1\neq\emptyset$. Moreover, since $A_{F_1}\subseteq F_1$, we have that $((G_{t_1}\cap G_{t_2})\setminus F_1)\setminus A_{F_1}=(G_{t_1}\cap G_{t_2})\setminus F_1\neq\emptyset$. In other words, $((G_{t_1}\cap G_{t_2})\setminus A_{F_1})\setminus(F_1\setminus A_{F_1})\neq\emptyset$. Thus $((G_{t_1}\setminus A_{F_1})\cap(G_{t_2}\setminus A_{F_1}))\setminus (F_1\setminus A_{F_1})\neq\emptyset$. Therefore, we conclude that $F_1\setminus A_{F_1}$ is a special leaf of $\Delta_{A_{F_1}}$.

By the induction hypothesis, the additivity of depth \cite[Theorem 3.1.34]{RHV}, Subclaim 1, Subclaim 2, and \Cref{contraction CM}, we have 
\begin{equation}\label{eq5}
    \begin{split}
        \d(R/((I(\Delta)^{[k]}+\l \x_{F_1}\r):\x_{A_{F_1}}))
        &=\d(\K[V(\D)]/(I(\D_{A_{F_1}})^{[k]}
        +\l \x_{F_1\setminus A_{F_1}}\r))\\
        &=\d(\K[V(\D_{A_{F_1}})]/(I(\D_{A_{F_1}})^{[k]}
        +\l \x_{F_1\setminus A_{F_1}}\r))\\
        & \hspace{140pt}  +\d(\K[A_{F_1}])\\
        &=|V(\Delta_{A_{F_1}})|-r+k-1+|V(\D)|-|V(\D_{A_{F_1}})|\\
        &=|V(\Delta)|-r+k-1.
    \end{split}
\end{equation}
Next, we have $I(\Delta)^{[k]}+\l \x_{F_1}\r+\l \x_{A_{F_1}}\r=I(\Delta_2)^{[k]}+\l \x_{A_{F_1}}\r$, where $\F(\Delta_2)=(\F(\Delta)\cup\{A_{F_1}\})\setminus \N_{\Delta}[F_1]$. Note that for each $H\in\F(\D_2)\cap\F(\D)$, we have $H\cap A_{F_1}=\emptyset$. Thus, one can see that $A_{F_1}$ is a special leaf of $\Delta_2$. Moreover, since $\Delta$ is Cohen-Macaulay, $\Delta\setminus \N_{\Delta}[F_1]$ is also Cohen-Macaulay. Hence, $\Delta_2$ is also a Cohen-Macaulay simplicial forest. Therefore, by the induction hypothesis, we have
\begin{equation}\label{eq6}
    \begin{split}
    \d(R/(I(\Delta)^{[k]}+\l \x_{F_1}\r+\l \x_{A_{F_1}}\r))&=|V(\Delta_2)|-r+k-1+|V(\Delta)|-|V(\Delta_2)|\\
    &=|V(\Delta)|-r+k-1.
\end{split}
\end{equation}
Now, we consider the following short exact sequence:
\[
0\rightarrow R/(J:\x_{A_{F_1}})\rightarrow R/J\rightarrow R/(J+\l \x_{A_{F_1}}\r)\rightarrow 0,
\]
where $J=I(\Delta)^{[k]}+\l \x_{F_1}\r$. Applying the Depth Lemma \cite[Lemma 2.3.9]{RHV}, we get 
\[
\d(R/(I(\Delta)^{[k]}+\l\x_{F_1}\r))\ge \min\left\{\d(R/(J:\x_{A_{F_1}})), \d(R/(J+\l \x_{A_{F_1}}\r))\right\}.
\]
Then by  \Cref{eq5} and (\ref{eq6}), we obtain
$\d(R/(I(\Delta)^{[k]}+\l\x_{F_1}\r))\ge |V(\Delta)|-r+k-1$. This completes the proof of the Claim. \par 

Again, observe that 
\begin{align*}
\d(R/(I(\Delta)^{[k]}+\l\x_{F_1}\r))\le \dim(R/(I(\Delta)^{[k]}+\l\x_{F_1}\r))\le \dim(R/I(\Delta)^{[k]}).    
\end{align*}
Thus, by \Cref{dimension lemma}, $\d(R/(I(\Delta)^{[k]}+\l\x_{F_1}\r))\leq|V(\Delta)|-r+k-1$. Consequently, 
\begin{align}\label{eq7}
    \d(R/(I(\Delta)^{[k]}+\l\x_{F_1}\r))=|V(\Delta)|-r+k-1.
\end{align}
This completes the proof of (ii).

Next, using \cref{colon lemma}, we have $(I(\Delta)^{[k]}:\x_{F_1})=I(\Delta_3)^{[k-1]}$,
    where $\F(\Delta_3)=\F(\Delta)\setminus \N_{\Delta}[F_1]$, and $\Delta_3$ is a Cohen-Macaulay simplicial forest. Thus, by the induction hypothesis, $\d(R/(I(\Delta)^{[k]}:\x_{F_1}))=|V(\Delta_3)|-(r-1)+(k-1)-1+|F_1|$. Note that $|V(\Delta_3)|=|V(\Delta)|-|F_1|$. Therefore, we have 
    \begin{align}\label{eq8}
     \d(R/(I(\Delta)^{[k]}:\x_{F_1}))=|V(\Delta)|-r+k-1.
    \end{align}
    Now, let us consider the following short exact sequence:
    $$0\rightarrow R/(I(\Delta)^{[k]}:\x_{F_1})\rightarrow R/I(\Delta)^{[k]}\rightarrow R/(I(\Delta)^{[k]}+\l\x_{F_1}\r) \rightarrow 0.$$
    Again, using the Depth Lemma, \Cref{eq7} and (\ref{eq8}), we get $\d(R/I(\Delta)^{[k]})\ge |V(\Delta)|-r+k-1$. Finally, since $\d(R/I(\Delta)^{[k]})\le \dim(R/I(\Delta)^{[k]})$, we conclude in view of \Cref{dimension lemma} that $\d(R/I(\Delta)^{[k]})=|V(\Delta)|-r+k-1=|V(\Delta)|-\nu(\D)+k-1$.
\end{proof}

%We now state and prove the main theorem of this article.
As an immediate consequence of \Cref{dimension lemma} and \Cref{main theorem 1}, we achieve our primary goal of the paper:

\begin{corollary}\label{main cor}
    Let $\Delta$ be a Cohen-Macaulay simplicial forest. Then for all $k\ge 1$, $R/I(\Delta)^{[k]}$ is also Cohen-Macaulay. 
\end{corollary}
\begin{proof}
    If $1\le k\le \nu(\Delta)$, then by \Cref{dimension lemma} and \Cref{main theorem 1}, we have that $R/I(\Delta)^{[k]}$ is Cohen-Macaulay. If $k>\nu(\D)$, then $I(\D)^{[k]}$ is the zero ideal, and hence, $R/I(\Delta)^{[k]}=R$ is Cohen-Macaulay. 
\end{proof}

\begin{remark}
    The result of the above corollary does not hold in general for the facet ideals of Cohen-Macaulay simplicial complexes. As simple graphs can be realized as $1$-dimensional simplicial complexes, consider the edge ideal of the graph $G_1$ given in \Cref{fig:2}. It follows from \cite{CM} that $R/I(G_1)$ is Cohen-Macaulay. On the other hand, $I(G_1)^{[2]}$ is generated by the monomials $\{ x_1y_1x_3y_3, x_1y_1x_2x_3,x_1y_1x_2y_2,x_3y_3x_2y_2,x_3y_3x_1x_2,x_2y_2x_1x_3\}$. Observe that both $\l x_1,x_2\r$ and $\l y_1,y_2,y_3\r$ are two minimal prime ideals containing $I(G_1)^{[2]}$. Thus $I(G_1)^{[2]}$ is not unmixed and consequently, $R/I(G_1)^{[2]}$ is not Cohen-Macaulay.
\end{remark}
\begin{figure}[H]
    \centering
    \begin{tikzpicture}
    [scale=.65]
            \path
            (0, 0)   coordinate (x1)
            (0.5, -1) coordinate (x2)
            ( -0.5, -1) coordinate (x3)
            (0, 0.75)   coordinate (y1)
            (1.2, -1.5)   coordinate (y2)
            (-1.2, -1.5)   coordinate (y3)
            (0, -1.5)   coordinate (D2);

            %\fill [lightgray]         (x1) -- (x2) -- (x3) -- (x1) -- cycle;

            \draw
            (x1) -- (x2) -- (x3) -- (x1) -- cycle
            (x1) -- (y1)
            (x2) -- (y2)
            (x3) -- (y3);

            \tikzstyle{vertex} = [draw,circle,fill=black,inner sep = 1.5pt]
            \node[vertex] [label = right: $x_1$] at (x1) {};
            \node[vertex] [label = right: $x_2$] at (x2) {};
            \node[vertex] [label = left: $x_3$] at (x3) {};
            \node[vertex] [label = above: $y_1$] at (y1) {};
            \node[vertex] [label = below: $y_2$] at (y2) {};
            \node[vertex] [label = below: $y_3$] at (y3) {};
            \node[] [label = below: $G_1$] at (D2) {};
        \end{tikzpicture}
        \hspace{3cm}
    \begin{tikzpicture}
    [scale=.65]
            \path
            (1.5, -1)   coordinate (x3)
            (0.5, -1) coordinate (x2)
            ( -0.5, -1) coordinate (x1)
            (-0.5, 0.35)   coordinate (y1)
            (0.5, 0.35)   coordinate (y2)
            (1.5, 0.35)   coordinate (y3)
            (0.5, -1.5)   coordinate (D2);

            %\fill [lightgray]         (x1) -- (x2) -- (x3) -- (x1) -- cycle;

            \draw
            (x1) -- (x2) -- (x3)
            (x1) -- (y1)
            (x2) -- (y2)
            (x3) -- (y3);

            \tikzstyle{vertex} = [draw,circle,fill=black,inner sep = 1.5pt]
            \node[vertex] [label = below: $x_1$] at (x1) {};
            \node[vertex] [label = below: $x_2$] at (x2) {};
            \node[vertex] [label = below: $x_3$] at (x3) {};
            \node[vertex] [label = above: $y_1$] at (y1) {};
            \node[vertex] [label = above: $y_2$] at (y2) {};
            \node[vertex] [label = above: $y_3$] at (y3) {};
            \node[] [label = below: $G_2$] at (D2) {};
        \end{tikzpicture}
    \caption{}
    \label{fig:2}
\end{figure}

\begin{remark}
For any simplicial complex $\Delta$ the square-free power $I(\D)^{[k]}$ is again a square-free monomial ideal, and thus one can associate a simplicial complex $\D'$ to the ideal $I(\D)^{[k]}$. It is interesting to note that although all square-free powers of a Cohen-Macaulay simplicial forest $\D$ is again Cohen-Macaulay (by \Cref{main cor}), it may happen that $I(\D)^{[k]}$ is not a simplicial forest for some $k$. For instance, consider the edge ideal $I(G_2)=\l x_1x_2, x_2x_3, x_1y_1, x_2y_2,x_3y_3\r$ of the graph $G_2$ as in \Cref{fig:2}. Observe that $I(G_2)$ is the edge ideal of a whisker three path and thus Cohen-Macaulay \cite{CM}. Note that, 
\[
I(G_2)^{[2]}=\l x_1y_1x_2y_2, x_1y_1x_3y_3, x_1y_1x_2x_3, x_1x_2x_3y_3, x_2y_2x_3y_3\r.
\] If we set $I(\D')=I(G_2)^{[2]}$, then it is easy to see that $\Delta'$ does not contain any free vertex. Thus $I(G_2)^{[2]}$ is not a facet ideal of a simplicial forest.
\end{remark}

In the rest of this section, we deal with the normalized depth function of square-free monomial ideals. Let $I\subseteq R$ be a square-free monomial ideal, and assume that $R$ is the smallest polynomial ring where $\mathcal{G}(I)\subseteq R$. Let $d_k$ denote the minimum degree of a monomial belonging to $\mathcal{G}(I^{[k]})$. The depth function of square-free powers of square-free monomial ideals was first considered in \cite{EHHM12} (see also \cite{FiHeHi}). In \cite{EHHM12}, the authors showed that for any $k\geq 0$, if $I^{[k]}\neq 0$, then one always have $\mathrm{depth}(R/I^{[k]})\geq d_k-1$. Thus, it makes sense to define the \textit{normalized depth function of }$I$ in the following way:
\[
g_I(k)=\mathrm{depth}(R/I^{[k]})- (d_k-1), \text{ for }1\leq k\leq \nu(I). 
\]

It was conjectured in \cite{EHHM12} that the normalized depth function $g_I(k)$ of any square-free monomial ideal $I$ is a nonincreasing function. However, in \cite{Fakhari_increasing}, Fakhari showed that for the class of monomial ideal $I=\l x_1x_3x_{i+4},x_1x_4x_5,x_2x_3x_4,x_2x_3x_6\mid i\in[n-4]\r$, $g_I(2)-g_I(1)$ can be arbitrarily large. Note that $I$ is not the facet ideal of a simplicial forest. Indeed, it is easy to see that the subcollection $\{\{x_1,x_3,x_5\},\{x_1,x_4,x_5\},\{x_2,x_3,x_4\}\}$ does not contain any leaf. In the following theorem, we show that when $\Delta$ is a Cohen-Macaulay simplicial forest, the normalized depth function $g_{I(\Delta)}(k)$ is nonincreasing.  
\begin{theorem}\label{normalised depth}
    Let $\Delta$ be a Cohen-Macaulay simplicial forest. Then, the normalized depth function $g_{I(\Delta)}(k)$ is nonincreasing.
\end{theorem}
\begin{proof}
    Note that $g_{I(\Delta)}(k)-g_{I(\Delta)}(k+1)=\d(R/I(\Delta)^{[k]})-\d(R/I(\Delta)^{[k+1]})+d_{k+1}-d_k$. Since $d_{k+1}-d_k\ge 1$, using \Cref{main theorem 1} we obtain $g_{I(\Delta)}(k)-g_{I(\Delta)}(k+1)\ge 0$, and thus $g_{I(\Delta)}(k)$ is a nonincreasing function. 
\end{proof}

\noindent
\textbf{Acknowledgements.} The authors would like to thank the anonymous referees for their careful reading and numerous valuable suggestions, which have significantly enhanced the clarity and readability of the paper. In particular, Remark 2.7, Remark 3.8, and Remark 3.9 arose due to the suggestions of the referees. They also thank Francesco Navarra and Ayesha Asloob Qureshi for carefully going through the first draft of this paper. Das and Roy are partially supported by a grant from the Infosys Foundation. Saha would like to thank the Infosys Foundation for a grant and the NBHM (India) for a Postdoctoral Fellowship during his stay at the Chennai Mathematical Institute.

%\noindent
%\section*{Acknowledgements}
%The authors would like to thank Francesco Navarra and Ayesha Asloob Qureshi for carefully going through the first draft of this paper. Das and Roy are supported by Postdoctoral Fellowships at Chennai Mathematical Institute. Saha would like to thank the National Board for Higher Mathematics (India) for the financial support through the NBHM Postdoctoral Fellowship. All the authors are partially supported by a grant from the Infosys Foundation.

%\subsection*{Data availability statement} Data sharing does not apply to this article as no new data were created or
%analyzed in this study.

%\subsection*{Conflict of interest} The authors declare that they have no known competing financial interests or personal
%relationships that could have appeared to influence the work reported in this paper.

\bibliographystyle{abbrv}
\bibliography{ref}

\begin{thebibliography}{10}

\bibitem{BHZN}
M.~Bigdeli, J.~Herzog, and R.~Zaare-Nahandi.
\newblock On the index of powers of edge ideals.
\newblock {\em Comm. Algebra}, 46(3):1080--1095, 2018.

\bibitem{CFS1}
M.~Caboara, S.~Faridi, and P.~Selinger.
\newblock {Simplicial cycles and the computation of simplicial trees.}
\newblock {\em J. Symb. Comput.}, 42(1-2):74 -- 88, 2007.

\bibitem{CowsikNori}
R.~C. Cowsik and M.~V. Nori.
\newblock On the fibres of blowing up.
\newblock {\em J. Indian Math. Soc. (N.S.)}, 40(1-4):217--222, 1976.

\bibitem{CFL1}
M.~Crupi, A.~Ficarra, and E.~Lax.
\newblock Matchings, square-free powers and {B}etti splittings.
\newblock {\em Illinois J. Math.}, 69(2):353--372, 2025.

\bibitem{DRS12024}
K.~K. Das, A.~Roy, and K.~Saha.
\newblock Square-free powers of {C}ohen-{M}acaulay forests, cycles, and whiskered cycles.
\newblock {\em arXiv:2409.06021}, 2024.

\bibitem{ErFi1}
N.~Erey and A.~Ficarra.
\newblock Matching powers of monomial ideals and edge ideals of weighted oriented graphs.
\newblock {\em J. Alg. and Appl.}, Online Ready:10.1142/S0219498826501185, 2025.

\bibitem{EHHS}
N.~Erey, J.~Herzog, T.~Hibi, and S.~Saeedi~Madani.
\newblock Matchings and squarefree powers of edge ideals.
\newblock {\em J. Combin. Theory Ser. A}, 188:Paper No. 105585, 24, 2022.

\bibitem{EHHM12}
N.~Erey, J.~Herzog, T.~Hibi, and S.~Saeedi~Madani.
\newblock The normalized depth function of squarefree powers.
\newblock {\em Collect. Math.}, 75(2):409--423, 2024.

\bibitem{ErHi1}
N.~Erey and T.~Hibi.
\newblock Squarefree powers of edge ideals of forests.
\newblock {\em Electron. J. Combin.}, 28(2):Paper No. 2.32, 16, 2021.

\bibitem{Faridi2002}
S.~Faridi.
\newblock The facet ideal of a simplicial complex.
\newblock {\em Manuscripta Math.}, 109(2):159--174, 2002.

\bibitem{Faridi2005}
S.~Faridi.
\newblock {Cohen-Macaulay properties of square-free monomial ideals}.
\newblock {\em J. Combin. Theory Ser. A}, 109(2):299--329, 2005.

\bibitem{FiHeHi}
A.~Ficarra, J.~Herzog, and T.~Hibi.
\newblock Behaviour of the normalized depth function.
\newblock {\em Electron. J. Combin.}, 30(2):Paper No. 2.31, 16, 2023.

\bibitem{MF2024}
A.~Ficarra and S.~Moradi.
\newblock Monomial ideals whose all matching powers are {C}ohen-{M}acaulay.
\newblock {\em arXiv:$2410.01666$}, 2024.

\bibitem{HHTX1}
J.~Herzog, T.~Hibi, N.~Trung, and X.~Zheng.
\newblock {Standard graded vertex cover algebras, cycles and leaves.}
\newblock {\em Trans. Am. Math. Soc.}, 360(12):6231 -- 6249, 2008.

\bibitem{HHZ04}
J.~Herzog, T.~Hibi, and X.~Zheng.
\newblock Monomial ideals whose powers have a linear resolution.
\newblock {\em Math. Scand.}, 95(1):23--32, 2004.

\bibitem{kamberi2024squarefreepowerssimplicialtrees}
E.~Kamberi, F.~Navarra, and A.~A. Qureshi.
\newblock On squarefree powers of simplicial trees.
\newblock {\em arXiv:$2406.13670$}, 2024.

\bibitem{Rauf2010}
A.~Rauf.
\newblock Depth and {S}tanley depth of multigraded modules.
\newblock {\em Comm. Algebra}, 38(2):773--784, 2010.

\bibitem{RTY11}
G.~Rinaldo, N.~Terai, and K.-i. Yoshida.
\newblock Cohen-{M}acaulayness for symbolic power ideals of edge ideals.
\newblock {\em J. Algebra}, 347:1--22, 2011.

\bibitem{Fakhari_increasing}
S.~A. Seyed~Fakhari.
\newblock {An increasing normalized depth function}.
\newblock {\em J. Commut. Algebra}, 16(4):497 -- 499, 2024.

\bibitem{SaS1}
S.~A. Seyed~Fakhari.
\newblock On the {C}astelnuovo-{M}umford regularity of squarefree powers of edge ideals.
\newblock {\em J. Pure Appl. Algebra}, 228(3):Paper No. 107488, 12, 2024.

\bibitem{SVV1994}
A.~Simis, W.~V. Vasconcelos, and R.~H. Villarreal.
\newblock On the ideal theory of graphs.
\newblock {\em J. Algebra}, 167(2):389--416, 1994.

\bibitem{CM}
R.~H. Villarreal.
\newblock Cohen-{M}acaulay graphs.
\newblock {\em Manuscripta Math.}, 66(3):277--293, 1990.

\bibitem{RHV}
R.~H. Villarreal.
\newblock {\em Monomial algebras}.
\newblock Monographs and Research Notes in Mathematics. CRC Press, Boca Raton, FL, second edition, 2015.

\end{thebibliography}

\end{document}